\documentclass[10pt,leqno]{article}
\usepackage{amsmath, amscd, amsthm, amssymb, graphics, xypic, mathrsfs}
\usepackage{fancyhdr}
\usepackage{eucal}

\DeclareFontFamily{OT1}{pzc}{}
\DeclareFontShape{OT1}{pzc}{m}{it}%
             {<-> s * [1.195] pzcmi7t}{}
\DeclareMathAlphabet{\mathscr}{OT1}{pzc}%
                                 {m}{it}

\usepackage[pagebackref=true]{hyperref}

\newcommand{\Spec}{\operatorname{Spec}}
\newcommand{\isomto}{{\;\stackrel{\sim}{\longrightarrow}\;}}

\renewcommand{\Vec}{\mathbf{Vec}}
\newcommand{\Filt}{\mathbf{Filt}}
\newcommand{\Rep}{\mathbf{Rep}}

\newcommand{\Q}{{\mathbb Q}}
\newcommand{\Z}{{\mathbb Z}}

\newcommand{\aone}{{{\mathbb A}^1}}
\newcommand{\pone}{{{\mathbb P}^1}}

\newcommand{\ga}{{{\mathbb G}_{a}}}
\newcommand{\gm}{{{\mathbb G}_{m}}}
\newcommand{\PGL}{{\mathrm{PGL}}}
\newcommand{\GL}{{\mathrm{GL}}}
\newcommand{\G}{{\mathrm{G}}}
\newcommand{\B}{{\mathrm{B}}}
\newcommand{\U}{{\mathrm{U}}}
\renewcommand{\H}{{\mathrm{H}}}
\renewcommand{\L}{{\mathrm{L}}}
\newcommand{\A}{{\mathrm{A}}}
\newcommand{\X}{{\mathrm{X}}}
\newcommand{\T}{{\mathrm{T}}}
\newcommand{\Lie}{{\mathrm{Lie}}}
\newcommand{\N}{{\mathcal N}}

\theoremstyle{plain}
\newtheorem{thm}{Theorem}[section]

\newtheorem{cor}[thm]{Corollary}
\newtheorem{prop}[thm]{Proposition}
\newtheorem{claim}[thm]{Claim}

\newtheorem*{thm*}{Theorem}
\newtheorem*{problem*}{Problem}

\theoremstyle{definition}
\newtheorem{defn}[thm]{Definition}

\newtheorem{notation}[thm]{Notation}

\theoremstyle{remark}
\newtheorem{rem}[thm]{Remark}
\newtheorem{ex}[thm]{Example}
\newtheorem{nex}[thm]{Non-example}

\numberwithin{equation}{section}

\newcommand{\shrinkmargins}[1]{
  \addtolength{\textheight}{#1\topmargin}
  \addtolength{\textheight}{#1\topmargin}
  \addtolength{\textwidth}{#1\oddsidemargin}
  \addtolength{\textwidth}{#1\evensidemargin}
  \addtolength{\topmargin}{-#1\topmargin}
  \addtolength{\oddsidemargin}{-#1\oddsidemargin}
  \addtolength{\evensidemargin}{-#1\evensidemargin}
  }

\shrinkmargins{.8}

\begin{document}
\pagestyle{fancy}
\renewcommand{\sectionmark}[1]{\markright{\thesection\ #1}}
\fancyhead{} \fancyhead[LO,RE]{\bfseries\footnotesize\thepage}
\fancyhead[LE]{\bfseries\footnotesize\rightmark}
\fancyhead[RO]{\bfseries\footnotesize\rightmark}
\lhead[odd]{\scriptsize\leftmark}
\rhead[\arabic{page}]{\footnotesize\arabic{page}} \chead[]{}
\cfoot[]{} \setlength{\headheight}{1cm}

\title{{\bf Equivariant sheaves on some spherical varieties}}
\author{Aravind Asok \thanks{Aravind Asok was partially supported by National Science Foundation Award DMS-0900813.}\\ \begin{footnotesize}Department of Mathematics\end{footnotesize} \\ \begin{footnotesize}University of Southern California\end{footnotesize} \\ \begin{footnotesize}Los Angeles, CA 90089 \end{footnotesize} \\ \begin{footnotesize}\url{asok@usc.edu}\end{footnotesize}\\
\and James Parson\\ \begin{footnotesize}Department of Mathematics\end{footnotesize} \\ \begin{footnotesize}Hood College\end{footnotesize} \\ \begin{footnotesize}Frederick, MD 21701\end{footnotesize} \\
\begin{footnotesize}\url{parson@hood.edu}\end{footnotesize}}

\maketitle

\section{Introduction}
In this announcement we describe categories of equivariant vector bundles on certain ``spherical" varieties. Our description is in linear-algebra terms: vector spaces equipped with filtrations, group and Lie-algebra actions, and linear maps preserving these structures.  The two parents of our description are i) the description of $\G$-equivariant vector bundles on homogeneous spaces $\G/\H$ as $\H$-representations and ii) Klyachko's description of the category of equivariant vector bundles on a toric variety ({\em cf.} \cite{Klyachko}) in terms of vector spaces equipped with families of filtrations satisfying a compatibility condition.  Results in this spirit were first presented by Kato ({\em cf.} \cite{Kato}), and we recover and extend his results in several directions. Detailed proofs of our results will appear in \cite{AsokParson}.

We begin with an illustrative example, whose description prefigures the general situation. Let $k$ be a field of characteristic $0$. Working in the category of $k$-schemes, let $X = \pone \times \pone$ equipped with the diagonal (left) action of $\G = \PGL_2$. Let $\T$ be the stabilizer in $\G$ of the point $x$ with bihomogeneous coordinates $([1:1],[0:1])$. Let $\Vec^\G(X)$ denote the category of $\G$-equivariant vector bundles on $\X$. Let $\Filt^\T(k)$ be the category whose objects are triples $(V,\rho,F^\bullet)$, where $V$ is a finite-dimensional $k$-vector space, where $\rho$ is a representation of $\T$ on $V$, and where $F^\bullet$ is a decreasing finite filtration on $V$. Morphisms in $\Filt^\T(k)$ are $\T$-module homomorphisms compatible with the filtrations.

\begin{thm}
\label{thm:example}
An appropriate restriction functor provides an equivalence between
\begin{itemize}
\item the category $\Vec^\G(X)$ of $\G$-equivariant vector bundles on $X$, and
\item the full subcateory $\Filt^\T(k)$ spanned by objects satisfying the compatibility condition
\begin{itemize}
\item[{\bf{(C)}}] the induced action of $\Lie(\T)$ on $V$ by $d\rho$ sends $F^i(V)$ to $F^{i-1}(V)$.
\end{itemize}
\end{itemize}
\end{thm}

\subsubsection*{Overview}
Section \ref{s:groupoids} fixes some notation.  Section \ref{s:example} contains an outline of our proof of Theorem \ref{thm:example}.  Section \ref{s:generalization} contains our proposed generalization of Theorem \ref{thm:example} to toroidal spherical varieties. The crucial points are Definition \ref{def:neutralizable} and Theorem \ref{thm:main}. Section \ref{s:remarks} indicates some specific cases to which our generalization applies, including those from \cite{Kato}.

\section{Groupoids, equivariance, and descent}
\label{s:groupoids}
We use the language of groupoids in schemes to describe our methods.  Let $k$ be a field. In what follows, all schemes will be $k$-schemes, and so we write \emph{$k$-scheme} and \emph{scheme} interchangeably. For a $k$-group scheme $\G$, a $\G$-scheme $X$ is a $k$-scheme equipped with a left $\G$-action.

\begin{defn}
  A $k$-groupoid scheme ${\mathcal G}$, or simply a {\em groupoid}, consists of two $k$-schemes ${\mathcal G}_0$ (the object scheme) and ${\mathcal G}_1$ (the morphism scheme), two morphisms $s,t: {\mathcal G}_1 \to {\mathcal G}_0$ called source and target, an identity map $e: {\mathcal G}_0 \to {\mathcal G}_1$, an inversion map $i: {\mathcal G}_1 \to {\mathcal G}_1$, and a composition map $m: {\mathcal G}_1 \times_{s,{\mathcal G}_0,t} {\mathcal G}_1 \to {\mathcal G}_1$, satisfying a collection of axioms.  We also write $p_1$ and $p_2$ for the two projections ${\mathcal G}_1 \times_{{\mathcal G}_0} {\mathcal G}_1 \to {\mathcal G}_1$.
\end{defn}

\begin{ex}[Group scheme]
Let $\mathcal{G}$ be a groupoid with $s = t$. The composition map $m$ along with the other groupoid structures then make $\mathcal{G}_1$ a group scheme over $\mathcal{G}_0$. We generally refer to a groupoid such as $\mathcal{G}$ simply as a group scheme.
\end{ex}

A morphism $f$ of $k$-groupoid schemes is a pair $(f_0,f_1)$ of morphisms $f_i:\mathcal{G}_i\to \mathcal{G}'_i$ of $k$-schemes that is compatible with the identities and composition.
We say that a morphism of groupoids $f: {\mathcal X} \to {\mathcal Y}$ is {\em fully faithful} if the square
\[
\xymatrix{
{\mathcal X}_1 \ar[r]^{f_1}\ar[d]^{(s,t)} & {\mathcal Y}_1 \ar[d]^{(s,t)} \\
{\mathcal X}_0 \times {\mathcal X}_0 \ar[r]^{(f_0,f_0)} & {\mathcal Y}_0 \times {\mathcal Y}_0
}
\]
is cartesian.

We will use the following groupoid constructions repeatedly.

\begin{ex}[Action groupoid]
For a $\G$-scheme $X$, we write $\G \ltimes X$ for the action groupoid associated with the $\G$-action on $X$.  The object scheme of $\G \ltimes X$ is $X$ itself, and the morphism scheme of $\G \ltimes X$ is $\G \times X$.  The source and target maps $\G \times X \to X$ are, respectively, the projection onto $X$ and the action morphism.  The groupoid composition $(\G \times X) \times_{X} (\G \times X) \to \G \times X$ is provided by the multiplication map $\G \times \G \times X \to \G \times X$.
\end{ex}

\begin{ex}[Semi-direct product groupoid]
The semi-direct-product (or diagonal) construction generalizes the above action groupoid construction. Let $\G$ be a group scheme, and let $\mathcal{X}$ be a groupoid in the category of $\G$-schemes. That is, $\mathcal{X}_0$ and $\mathcal{X}_1$ are equipped with actions of $\G$, and all of the morphisms defining the groupoid structure are $\G$-equivariant. We then have a groupoid $\G\ltimes \mathcal{X}$ with object scheme $\mathcal{X}_0$ and morphism scheme $\G\times \mathcal{X}_1$. The source morphism $G\times \mathcal{X}_1\to \mathcal{X}_0$ is the composition of projection onto $\mathcal{X}_1$ with the source morphism of $\mathcal{X}$. The target morphism is the composition of the $\G$-action $\G\times \mathcal{X}_1\to\mathcal{X}_1$ with the target morphism of $\mathcal{X}$. Composition of morphisms in $\G\ltimes \mathcal{X}$ is defined as for semi-direct products of groups: for $(g,x)$ and $(h,y)$ in $\G\times \mathcal{X}_1$ with $s(x) = t(hy) = ht(y)$, we let $(g,x)\circ (h,y) = (gh,h^{-1}x\circ y)$.
\end{ex}

\begin{ex}[Induced groupoid]
 Let ${\mathcal X}$ be a $k$-groupoid scheme, and let $i_0: Z \to {\mathcal X}_0$ be a morphism.  There is a {\em groupoid induced by ${\mathcal X}$ and $i_0$}, consisting of a groupoid ${\mathcal Z}$ and a morphism of groupoids $i: {\mathcal Z} \to {\mathcal X}$ (where ${\mathcal Z}_0 = Z$) with the property that $i$ is fully faithful.  More explicitly, define ${\mathcal Z}_1$ as the pullback of the diagram
\[
(Z \times Z) \stackrel{i_0 \times i_0}{\longrightarrow} {\mathcal X}_0 \times {\mathcal X}_0 \stackrel{(s,t)}{\longleftarrow} {\mathcal X}_1;
\]
the groupoid structures are inherited from those of ${\mathcal X}$. We also write $\mathcal{X}|_Z$ or $\mathcal{X}|_{i_0}$ for this induced groupoid.
\end{ex}

\begin{defn}
A morphism $f: {\mathcal X} \to {\mathcal Y}$ of $k$-groupoid schemes is an {\em (fpqc)-equivalence} if it is fully faithful and locally essentially surjective in the fqpc topology.
\end{defn}

The ``locally essentially surjective'' condition means the following: let $\mathcal{X}_0' = \mathcal{X}_0 \times_{f_0,s} \mathcal{Y}_1$, which parameterizes pairs $(x,\phi:f_0(x)\to y)$ of an object $x$ of $\mathcal{X}$ and a morphism $\phi:f_0(x)\to y$ in $\mathcal{Y}$. The composition of the projection $\mathcal{X}_0'\to \mathcal{Y}_1$ with the target morphism $\mathcal{Y}_1\to \mathcal{Y}_0$ defines a morphism $f_0':\mathcal{X}_0'\to \mathcal{Y}_0$. To say that $f$ is locally essentially surjective is to say that there is an fpqc cover of $\mathcal{Y}_0$ over which $f_0'$ has a section, i.e., fpqc locally on $\mathcal{Y}_0$, every object is isomorphic to the image of an object of $\mathcal{X}_0$. In the cases that we consider below either $f_0'$ itself will have a section, or it will have a section Zariski locally on $\mathcal{Y}_0$.

\begin{ex}
\label{ex:transitive}
Let $\mathcal{X}$ be a groupoid, and let $x:\Spec(k)\to \mathcal{X}_0$. Note that $\mathcal{X}|_x$ is a ($k$-scheme) groupoid over $\Spec(k)$ and thus simply a group scheme over $\Spec(k)$, namely $\H \ltimes\Spec(k)$, where $\H$ is the stabilizer of $x$ in $\mathcal{X}$. Suppose that $\mathcal{X}$ is transitive in the sense that $t:s^{-1}(x) \to \mathcal{X}_0$ has a section fpqc locally on $\mathcal{X}_0$, i.e.,  every object in $\mathcal{X}$ is fpqc locally isomorphic to $x$. Then $\H\ltimes \Spec(k)\to \mathcal{X}$ is an equivalence by definition.
%
\end{ex}

The usual definition of equivariance for a vector bundle on a scheme equipped with a group action can be extended to equivariance for groupoids.

\begin{defn}
An \emph{equivariant vector bundle} on a groupoid $\mathcal{X}$ is a pair $(V,\varphi)$, where $V$ is a vector bundle on $\mathcal{X}_0$, and where $\varphi:s^*V\to t^*V$ is an isomorphism of vector bundles over $X_1$ that satisfies the cocycle condition $m^*\varphi = p_1^*\varphi \circ p_2^*\varphi$.
\end{defn}

A morphism $f:(V,\varphi)\to (V',\varphi')$ of equivariant vector bundles is a morphism $f:V\to V'$ of vector bundles on $\mathcal{X}_0$ satisfying $t^*f \circ\varphi = \varphi' \circ s^*f$.

\begin{notation}
The category of {\em equivariant vector bundles on ${\mathcal X}$} (or synonymously the category of {\em representations of ${\mathcal X}$}), denoted $\Vec(\mathcal{X})$, is the category whose objects are equivariant vector bundles on $\mathcal{X}$ and whose morphisms are as above.  Furthermore, we write $\Vec^\G(X)$ for $\Vec(\G\ltimes X)$ and $\Rep(\G)$ for $\Vec(\G\ltimes\Spec(k))$.
\end{notation}

Equivariant vector bundles pull back along morphisms of groupoids: let $f: {\mathcal X} \to {\mathcal X}'$ be a morphism of groupoids.  Given a representation $(V',\varphi')$ of ${\mathcal X}'$, we have the representation $f^*(V',\varphi') = (f_0^*V',f_1^*(\varphi'))$ of $\mathcal{X}$.  One can then apply descent techniques ({\em cf.}  \cite{Giraud} and \cite{SGA1}) to study the properties of such pullback functors of representations of groupoids; we use the following descent result:

\begin{prop}
\label{prop:descent}
 Let $f: {\mathcal X} \to {\mathcal X}'$ be an equivalence of groupoids.  The induced functor $f^*: \Vec({\mathcal X}') \to \Vec({\mathcal X})$ is fully faithful and essentially surjective.
\end{prop}

\begin{rem}
The reader familiar with the theory of stacks will have no trouble reformulating our results in that language.  For example, fqpc equivalences of groupoids induce equivalences of associated stacks, and equivalences of stacks induce equivalences of associated categories of representations.
\end{rem}

\section{Equivalences of groupoids and local structure}
\label{s:example}
For this section, let $\G = \PGL_2$, and $X = \pone \times \pone$ equipped with the diagonal (left) $\G$-action.
To prove Theorem \ref{thm:example}, we will construct a groupoid ${\mathcal Y}$ equivalent to the action groupoid $\G \ltimes X$, but whose representations are easier to describe. The groupoid $\mathcal{Y}$ will decompose as a semi-direct product $\gm \ltimes \N$ for a $\gm$-equivariant group scheme $\N$ over $\aone$. This semi-direct-product decomposition allows us to describe representations of $\mathcal{Y}$ as $\gm$-equivariant representations of $\N$. We analyze such representations using a simple description of $\gm$-equivariant vector bundles on $\aone$ along with a bit of equivariant Lie theory for $\N$.

Let $a_{ij}$ ($1 \leq i$,$j \leq 2$) be homogeneous coordinates on $\G$ and and let $([x_0:x_1],[y_0:y_1])$ be bihomogeneous coordinates on $\pone \times \pone$.  Let $\B$ be the Borel subgroup of $\G$ defined by $a_{21} = 0$.  Let $x$ be the standard coordinate on $\aone$, and let $t$ be the standard coordinate on $\gm$.

\subsubsection*{Simplification: orbit structure and slicing}
Before we start our simplification of $G\ltimes X$, we adjust our groupoid by restricting to an open subscheme of $X$ adapted to $\B$. In our generalization below, the theory of spherical varieties dictates this preliminary step.

Let $X_\B$ be the open subscheme $\aone \times \aone$ of $X$ defined by the non-vanishing of $x_1$ and of $y_1$.  This subscheme $X_\B$ is stable under the action of $\B$. Let ${\mathcal X}$ be the groupoid induced by $\G \ltimes X$ and the open immersion $X_\B \to X$.  Explicitly, ${\mathcal X}_0 = X_\B$, and ${\mathcal X}_1$ is the open subscheme of $\G \times X$ defined by the non-vanishing of $x_1, y_1, a_{21}x_0+a_{22}x_1,$ and $a_{21}y_0 + a_{22}y_1$.  The additional groupoid structures are inherited from $\G \ltimes X$. The morphism $\mathcal{X}\to \G\ltimes X$ is fully faithful by construction, and it is Zariski-locally essentially surjective, since the $\G(k)$-translates of $X_B$ cover $X$. Thus we have:

\begin{prop}
\label{prop:firstequiv}
The morphism of groupoids ${\mathcal X} \to \G \ltimes X$ is an equivalence.
\end{prop}

\paragraph*{Orbit-structure morphism.} The first piece of our simplification system is the \emph{orbit-structure morphism}: consider the morphism of schemes $X_\B \to \aone$ defined by the function $f_0 = \frac{x_0}{x_1} - \frac{y_0}{y_1}$.  We can extend this morphism to a morphism of groupoids $f: {\mathcal X} \to \gm \ltimes \aone$ by letting $f_1$ correspond to
\[
\left(
\frac{a_{11}a_{22}-a_{12}a_{21}}{(a_{21}\frac{x_0}{x_1}+a_{22})(a_{21}\frac{y_0}{y_1}+a_{22})},\frac{x_0}{x_1} - \frac{y_0}{y_1}
\right).
\]
We call $f: {\mathcal X} \to \gm \ltimes \aone$ the {\em orbit-structure morphism} since it classifies the $\G$-orbits on $X$ in terms of the $\gm$ orbits in $\aone$. Note that $X$ decomposes as the union of the ($\G$-stable) diagonal $\pone$ closed subscheme and its ($\G$-stable) complementary open subscheme. Each of these subschemes is homogeneous for the action of $\G$, and their intersections with $X_\B$ are the vanishing and non-vanishing loci of $f_0$, respectively. The orbit-structure morphism $(f_0,f_1)$ captures more about the action of $\G$ on $X$: for example, for any test scheme $Z$, two morphisms $x,y:Z\to X_\B$ are equivalent under $\G$ if and only if $f_0(x),f_0(y)\in\Gamma(Z,\mathcal{O}_Z)$ differ by an element of $\Gamma(Z,\mathcal{O}_Z)^\times$.

\paragraph*{Slicing.} The second piece our our simplification system is the \emph{slice}: consider the morphism of schemes $\aone \to X_\B \subset X$ defined by $\frac{x_0}{x_1} \mapsto x$ and $\frac{y_0}{y_1} \mapsto 0$, which is a \emph{slice} for the $\G$ action on $\X$. Let ${\mathcal Y}$ be the groupoid induced by $\G \ltimes X$ and the slice morphism $\aone \to X$. More explicitly, ${\mathcal Y}_0 = \Spec(k[x])$ and ${\mathcal Y}_1$ is the subscheme of $\G \times \aone$ where $a_{12} = 0$, and where $a_{11}$, $a_{22}$ and $\frac{a_{21}}{a_{11}}x + \frac{a_{22}}{a_{11}}$ are invertible. Thus $\mathcal{Y}_1 = \Spec(k[\gamma,\delta,x,\delta^{-1},(\gamma x + \delta)^{-1}])$ with $\gamma = \frac{a_{21}}{a_{11}}$ and $\delta = \frac{a_{22}}{a_{11}}$. The source morphism $\mathcal{Y}_1\to \mathcal{Y}_0 = \Spec(k[x])$ corresponds to $x\mapsto x$, and the target morphism corresponds to $x\mapsto (\gamma x+ \delta)^{-1}x$.

Let $\U$ be the subgroup of $\B$ defined by $a_{11} = a_{22}$. One finds by an easy computation that the action of $\U$ on $X_\B$ defines an isomorphism $\U\times\aone \to X_\B$, and so $\mathcal{Y}\to \mathcal{X}$ is essentially surjective. Combining this observation with Proposition \ref{prop:firstequiv}, we find:

\begin{prop}
The morphism ${\mathcal Y} \to \G \ltimes X$ is an equivalence.
\end{prop}

The groupoid $\mathcal{Y}$ is what we will reduce to a semi-direct product. To make this reduction, we enrich the slice $\aone\to X$ to a groupoid morphism: let $\gm \to \G$ be the homomorphism given by $a_{11} \mapsto t$, $a_{12} \mapsto 0$, $a_{21} \mapsto 0$, $a_{22} \mapsto 1$.  The morphism $\aone\to X$ intertwines the standard action of $\gm$ on $\aone$ with the action of $\gm$ on $X$ via $\gm \to \G$. Thus we have a groupoid morphism $\gm\ltimes\aone\to \G\ltimes X$. By the construction of $\mathcal{Y}$, this morphism factors through $s:\gm \ltimes \aone \to \mathcal{Y}$. This enhanced slice morphism $s$ is a section of the orbit-structure morphism: $fs$ is the identity on $\gm \times \aone$.

\paragraph*{Splitting.} We split $\mathcal{Y}$ as a semi-direct product by combining the orbit-structure morphism $f$ and slice morphism $s$. The normal groupoid $\N$ in the semi-direct product is the kernel of the orbit-structure morphism: the restriction $f|_{\mathcal{Y}}$ is the identity on $\mathcal{Y}_0$, and on $\mathcal{Y}_1$, it corresponds to $t\mapsto (\gamma x + \delta)^{-1}$. Let $\N$ be the kernel of $f|_{\mathcal{Y}}$, so that $\N_0=\mathcal{Y}_0=\aone$, and $\N_1$ is cut out in ${\mathcal Y}_1$ by the equation $\gamma x + \delta = 1$. The groupoid structure on $\mathcal{Y}$ restricts to an $\aone$-group-scheme structure on $\N_1$. More explicitly, $\N_1 = \Spec(k[\gamma,\delta,x,\delta^{-1}]/(\gamma x + \delta -1))$. The group law on $\N_1$ is $(\gamma,\delta,x) \times (\gamma',\delta',x) = (\gamma + \delta \gamma',\delta \delta',x)$. The group scheme $\N_1$ is smooth over $\aone$ with connected fibers: the fiber over $x=0$ is $\ga$, and over $\aone\setminus 0$, the character $\delta$ defines an isomorphism from $\N_1$ to $\gm\times (\aone\setminus 0)$.

The groupoid $\N$ acquires a $\gm$-equivariant structure from $s$. The base $\N_0 = \aone$ has the standard action of $\gm$. To define the action of $\gm$ on $\N_1$, we use the identification $\gm\times \N_1 = (\gm \times \aone)\times_\aone \N_1$. The slice $s$ allows us to view $\gm \times \aone = (\gm \ltimes \aone)_1$ as morphisms in $\mathcal{Y}$, and conjugation in $\mathcal{Y}$ defines the action morphism $(\gm \times \aone)\times_\aone \N_1 \to \N_1$. After unwinding these definitions, one finds that the action $\gm\times \N_1\to \N_1$ is given by $\gamma\mapsto t^{-1}\gamma$, $\delta\mapsto \delta$, and $x\mapsto tx$.

Finally, $f$ and $s$ provide an isomorphism $\gm \ltimes \N\to \mathcal{Y}$. The morphism $(\gm \ltimes \N)_0 = \N_0 = \aone \to \mathcal{Y}_0 = \aone$ is the identity. The (iso)morphism $(\gm \ltimes \N)_1 = \gm \times \N_1 = (\gm\times \aone)\times_\aone \N_1 \to \mathcal{Y}_1$ is composition in $\N$. More explicitly, this isomorphism corresponds to the $k$-algebra homomorphism
\[
k[\gamma,\delta,x,\delta^{-1},(\gamma x + \delta)^{-1}]\to k[t,t^{-1}]\otimes_k k[\gamma,\delta,x]/(\gamma x + \delta - 1) = k[t,t^{-1},\gamma,\delta,x]/(\gamma x + \delta - 1)
\]
defined by $x\mapsto x$, $\gamma\mapsto t^{-1}\gamma$, and $\delta\mapsto t^{-1}\delta$. In summary, we have:

\begin{prop}
The orbit-structure and slice morphisms define a groupoid isomorphism $\gm \ltimes \N \isomto {\mathcal Y}$.
\end{prop}

\begin{cor}
\label{cor:simplification}
Pullback along the equivalence of groupoids $\gm \ltimes \N \to \G \ltimes X$ induces an equivalence of the corresponding categories of representations.
\end{cor}

\subsubsection*{Concretizing: combinatorial input}
Corollary \ref{cor:simplification} reduces the problem of describing $\G$-equivariant vector bundles on $X$ to a corresponding problem for $\gm\ltimes \N$. The semi-direct-product decomposition of $\gm\ltimes \N$ allows one to describe a representation of $\gm\ltimes \N$ as a $\gm$-equivariant vector bundle on $\aone$ equipped with a compatible action of $\N$.  This category of $\gm$-equivariant vector bundles on $\aone$ has a nice description that goes back to Rees.

\begin{thm}[{{\em cf.} \cite[Theorem 3.1]{Asok}}]
\label{thm:rees}
The functor ``fiber over $1$" induces an equivalence between the categories $\Vec^{\gm}(\aone)$ and the category $\Filt(k)$ of finite-dimensional $k$-vector spaces equipped with a finite decreasing filtration.
\end{thm}

Composing the restriction functor along $s:\gm\ltimes \aone\to \gm\ltimes \N$ with the functor of Theorem \ref{thm:rees} defines a functor $\Vec(\gm\ltimes\N)\to \Filt(k)$. On the other hand, we have $(\gm\ltimes \N)|_1 = (\G\ltimes X)|_x = \T\ltimes \Spec(k)$, and so restriction to the fiber at $1$ defines a functor $\Vec(\gm\ltimes\N)\to \Rep(\T)$. Combining these two constructions, we obtain a functor
\[
\Phi: \Vec(\gm\ltimes \N) \to \Filt^\T(k).
\]
By Corollary \ref{cor:simplification}, the following claim implies Theorem \ref{thm:example}.

\begin{claim}
The functor $\Phi: \Vec(\gm\ltimes\N) \to \Filt^\T(k)$ constructed above is fully faithful, and its essential image is spanned by those objects satisfying \textbf{(C)}.
\end{claim}

\begin{proof}[Sketch of proof]
In what follows, we identify $\mathcal{Y}$ and $\gm\ltimes \N$ using the isomorphism of the previous section.
\vskip 5pt
\noindent{\em Step 1.}  Objects in the image of $\Phi$ satisfy condition \textbf{(C)}.  The fiber of $\N$ at $1$ is $\T$.  Since $\N$ is smooth, $\Lie(\N)$ is a $\gm$-equivariant vector bundle on $\aone$.  We identify the fiber of $\Lie(\N)$ at $1$ with $\Lie(\T)$.  The formula for the $\gm$ action on $\N_1$ shows that Theorem \ref{thm:rees} equips this $1$-dimensional vector space with a filtration that has jump in degree $-1$, i.e., $F^i(\Lie(\T)) = 0$ for $i \geq 0$ and $F^i(\Lie(\T)) = \Lie(\T)$ for $i < 0$.

Let $W$ be an object of $\Vec(\gm\ltimes \N)$, which we view as a $\gm$-equivariant vector bundle on $\aone$ equipped with a $\gm$-equivariant representation of $\N$. Let $V = \Phi(W)$. The $\gm$-equivariance of the action of $\N$ on $W$ implies that the induced action of $\Lie(\T)$ on $V$ is compatible with the filtrations, i.e., for all $i,j$ one has $F^i(\Lie(\T)) \times F^j(V) \to F^{i+j}(V)$; since $F^\bullet$ has a single jump at $-1$, this compatibility is precisely \textbf{(C)}.
\vskip 5pt
\noindent{\em Step 2.} The functor $\Phi$ is fully faithful.  Since $\aone\setminus 0$ is dense in $\aone$, the restriction functor $\Vec(\mathcal{Y})\to \Vec(\mathcal{Y}|_{\aone\setminus 0})$ is faithful. Since $\mathcal{Y}|_{\aone\setminus 0}$ is transitive, the restriction functor $\Vec(\mathcal{Y}|_{\aone\setminus 0})\to \Rep(\T)$ is an equivalence ({\em cf.} Example \ref{ex:transitive}). Thus $\Vec(\mathcal{Y})\to \Rep(\T)$ is faithful. Since $\Filt^\T(k)\to \Rep(\T)$ is also faithful, so is $\Phi:\Vec(\mathcal{Y})\to \Filt^\T(k)$. Let $W,W'$ be objects of $\Vec(\mathcal{Y})$, and let $f:\Phi(W)\to\Phi(W')$ be a morphism in $\Filt^\T(k)$. Since $\Vec(\mathcal{Y}|_{\aone\setminus 0})\to \Rep(\T)$ is an equivalence, the morphism of $\T$-modules underlying $f$ lifts to a morphism $g:W|_{\aone\setminus 0}\to W'|_{\aone\setminus 0}$. This morphism extends $\mathcal{Y}$-equivariantly to $\aone$ precisely when the underlying morphism of $\gm$-equivariant vector bundles extends to $\aone$. By Theorem \ref{thm:rees}, the vector bundle morphism extends, since $f$ respects the filtrations.
\vskip 5pt
\noindent{\em Step 3.} The essential image of $\Phi$ is spanned by objects satisfying {\textbf{(C)}}.  Let $(V,\rho,F^\bullet)$ be an object of $\Filt^\T(k)$ satisfying {\textbf{(C)}}. Since $\Vec(\mathcal{Y}|_{\aone\setminus 0})\to \Rep(\T)$ is an equivalence, there is an object $W^o$ of $\Vec(\mathcal{Y}|_{\aone\setminus 0})$ whose fiber over $1$ is isomorphic to $(V,\rho)$ in $\Rep(T)$. By Theorem \ref{thm:rees}, the filtration $F^\bullet$ on $V$ determines a $\gm$-equivariant extension $W$ over $\aone$ of the $\gm$-equivariant vector bundle underlying $W^o$. Condition \textbf{(C)} implies that the action of $\Lie(\N|_{\aone\setminus 0})$ on $W^o$ extends to a $\gm$-equivariant action of $\Lie(\N)$ on $W$. By a Lie-theory argument, the existence of such an extension of the action to $\Lie(\N)$ implies that the action of $\N|_{\aone\setminus 0}$ on $W^o$ extends to an action of $\N$ on $W$.
\end{proof}

\section{A generalization to toroidal spherical varieties}
\label{s:generalization}
In this section, we formulate a generalization of Theorem \ref{thm:example}.  Henceforth, we take $k$ to be an algebraically closed field of characteristic $0$.

\subsubsection*{Toroidal spherical varieties}
The hypotheses on $k$ simplify statements about spherical varieties. Let $\G$ be a connected reductive group over $k$.  Fix a Borel subgroup $\B$ of $\G$. Let $X$ be a normal $\G$-variety with a marked point $x\in X(k)$. The marked $\G$-variety $(X,x)$ is a {\em spherical $(\G,\B)$-variety}, if the $\B$-orbit of $x$ in $X$ is dense (and hence open). If $(X,x)$ is spherical, then the $\G$-orbit of $x$ is also dense; we then write $X^o$ for the $\G$-orbit of $x$ in $X$ and $X^o_\B$ for the $\B$-orbit of $x$ in $X$.


From now on, we take $(X,x)$ to be a spherical $(\G,\B)$-variety. Let $\Delta_X$ be the closed subset of $X$ that is the union of the $\B$-stable prime divisors (irreducible codimension-$1$ closed subsets) of $X$ that are not $\G$-stable, and set $X_\B = X \setminus \Delta_X$.

\begin{defn}
 The spherical $(\G,\B)$-variety $(X,x)$ is {\em toroidal} if every $\B$-stable prime divisor that contains a $\G$-orbit is $\G$-stable, i.e.,  if the $\G$-saturation of $X_\B$ in $X$ is $X$.
\end{defn}

\begin{ex}
Let $\A$ be a torus over $k$. For our purposes, a \emph{toric $\A$-variety} is a spherical $(\A,\A)$-variety $(S,s)$ such that the orbit map $\A\to S$ associated to $s$ is an open immersion, i.e.,  such that $X^o = \A$. Toric varieties are the simplest toroidal spherical varieties.
\end{ex}

A basic invariant of a spherical variety is its \emph{weight lattice}: let $\H$ be the stabilizer of $x$ in $\G$. The weight lattice $\Lambda$ of $(X,x)$ is the character lattice of $\B/(\H\cap \B)\U$, where $\U$ is the unipotent radical of $\B$. By definition, $\Lambda$ is a subgroup of the character lattice of $\B$.

Note that $\Lambda\subset \mathrm{X}^*(\B)$ depends only on the $(\G,\B)$-spherical variety $(X^o,x)$ and not on the boundary structure. The classification of spherical varieties ({\em cf.} \cite{KnopLV}) associates to a toroidal $(\G,\B)$-spherical variety $(X,x)$ a fan $\Sigma$ in ${\rm Hom}_\Z(\Lambda,\Q)$. The fan reflects faithfully the structure of the embedding $X^o\to X$ in the sense that one can reconstruct $X$ from $X^o$ and $\Sigma$.

\subsubsection*{Simplification}
We now assume that our spherical $(\G,\B)$-variety $(X,x)$ is toroidal, and we go through the steps of \S\ref{s:example} for $(X,x)$. Let ${\mathcal X}= \G \ltimes X|_{X_\B}$. Since by the toroidal hypothesis the $\G$-saturation of $X_\B$ in $X$ is all of $X$, we have:

\begin{prop}
The morphism of groupoids ${\mathcal X} \to \G \ltimes X$ is an equivalence.
\end{prop}

\paragraph*{Orbit-structure morphism.} Let $(S,s)$ be a toric variety for a torus $\A$ such that the fan associated to $\A$ and $(S,s)$ is the fan $\Sigma$ associated to $(X,x)$. Note, in particular, that the character lattice of $\A$ is the weight lattice $\Lambda$ of $(X,x)$. By refining the construction of $\Sigma$ in \cite{KnopLV}, we produce a groupoid morphism ${\mathcal X} \to \A \ltimes S$ mapping $x\in X(k)$ to $s\in S(k)$. (This morphism can be characterized by a reasonable universal-mapping property.)  We call this morphism the {\em orbit-structure morphism} for the toroidal spherical $(\G,\B)$-variety $(X,x)$ or, when the context is clear, for $X$; the morphism to $\gm \ltimes \aone$ of \S\ref{s:example} is an instance of this construction.

\paragraph*{Slicing.} Let $\mathrm{P}\supset \B$ be the subgroup scheme of $\G$ normalizing the reduced subscheme supported on $\Delta_X$. Since $k$ has characteristic $0$, the $k$-scheme $\mathrm{P}$ is smooth. Since $\mathrm{P}$ contains $\B$, it is thus a parabolic subgroup of $\G$. Let $\mathrm{U}$ be the unipotent radical of $\mathrm{P}$. The local structure theory for toroidal spherical varieties ({\em cf.} \cite{KnAsymp}) produces a closed immersion $(S,s)\to (X_\B,x)$ and a Levi factor $\L$ of $\mathrm{P}$ that stabilizes the image of $S$. Furthermore, the action morphism $\mathrm{U}\times S\to X_\B$ is an isomorphism. Consequently, we have:

\begin{prop}
   Let ${\mathcal Y}$ be the groupoid induced by ${\mathcal X}$ and $S \to X_\B$. Then the groupoid homomorphism ${\mathcal Y} \to {\mathcal X}$ is an equivalence.
\end{prop}

Since $\L$ stabilizes the image of $S$, we have a groupoid morphism $\L\ltimes S\to \mathcal{Y}$. The composition $\L\ltimes S\to \mathcal{Y}\to \A\ltimes S$ with the orbit-structure morphism is the identity on the common object scheme $S$, and on the morphism schemes it corresponds to a group-scheme epimorphism $\L\to \A$. In the example of \S\ref{s:example}, we have $\L = \A = \gm$, and $\L\ltimes S\to \A\ltimes S$ is an isomorphism, which allows us to split $\mathcal{Y}$ as a semi-direct product. In the general case, without extra hypotheses, we cannot decompose ${\mathcal Y}$ as a semi-direct product as in the example of \S\ref{s:example}. In order to produce the decomposition, we must have a splitting of the homomorphism $\L\to \A$. By composing a splitting $\A\ltimes S\to \L\ltimes S$ with the enhanced slice $\L\ltimes S\to \mathcal{Y}$, we obtain a section $\A\ltimes S\to \mathcal{Y}$ of the orbit-structure morphism $\mathcal{Y}\to \A\ltimes S$.

A splitting of $\L\to \A$ exists if and only if for any maximal torus $\T$ of $\L$, the homomorphism $\T\to\A$ splits. Equivalently, a splitting exists if and only if the weight lattice $\Lambda$ of $(X,x)$, which is the character lattice of $\A$, is a direct summand of the character lattice of $\B$.

\begin{defn}\label{def:neutralizable}
A toroidal spherical $(\G,\B)$-variety $(X,x)$ is {\em neutralizable} if its weight lattice $\Lambda$ is a direct summand of $\mathrm{X}^*(\B)$.
\end{defn}

Note, in particular, that whether a given toroidal spherical $(\G,\B)$-variety $(X,x)$ is neutralizable depends only on the $\G$-homogeneous toroidal spherical $(\G,\B)$-variety $(X^o,x)$.

\paragraph*{Splitting.} Assume now that $X$ is a neutralizable, and fix a splitting $\A\to \L$ of $\L\to \A$. Let $\N$ be the kernel of the restriction of the orbit-structure morphism to $\mathcal{Y}$. As in \S\ref{s:example}, the groupoid $\N$ is a group scheme, and it comes equipped with an action of $\A$.

\begin{prop}
\label{prop:nsmooth}
  The $\A$-equivariant group scheme $\N$ over $S$ is smooth and with connected fibers.
\end{prop}

By a construction as in \S\ref{s:example}, the inclusion $\N\to \mathcal{Y}$ extends to a groupoid morphism $\A\ltimes \N\to \mathcal{Y}$.

\begin{prop}
  For a (split) neutralizable toroidal spherical $(\G,\B)$-variety $(X,x)$, the morphism of groupoids $\A \ltimes \N \to {\mathcal Y}$ is an isomorphism.
\end{prop}

\begin{rem}
  We use the term ``neutralizability" to suggest the following, which our work makes precise: the orbit-structure morphism makes the groupoid ${\mathcal X}$ (or, up to equivalence, $\G\ltimes X$) a gerbe over $\A\ltimes S$. The slice construction from the local structure theory of spherical varieties can sometimes be used to construct neutralizations (i.e.,  sections) of this gerbe. The choice of a neutralization allows one to interpret representations of ${\mathcal X}$ (or, equivalently, representations of $\G \ltimes X$) as representations of $\A \ltimes S$ equipped with a compatible (i.e., $\A$-equivariant) representation of $\N$.
\end{rem}

\subsubsection*{Concretization}
Under the neutralizability hypothesis, as in \S\ref{s:example}, the equivalence $\A\ltimes \N\to \G\ltimes X$ reduces the description of $\G$-equivariant vector bundles on $X$ to the description of equivariant vector bundles on $\A\ltimes \N$. As before, the first step is to describe the category of representations of $\A \ltimes S$; Klyachko provided such a description ({\em cf.} \cite{Klyachko}).

Let $\Sigma(1)$ be the set of $1$-dimensional cones in $\Sigma$ (which correspond to codimension-$1$ orbits of $\A$ in $S$). Thus each $\alpha\in\Sigma(1)$ is a ray in $\Q\otimes \X_*(\A)$. For each $\alpha\in\Sigma(1)$, let $n_\alpha\in \X_*(\A)$ be the primitive generator of $\alpha$. Let $\Filt^{\Sigma(1)}(k)$ be the category whose objects are collections $(V,\{F_\alpha^\bullet\}_{\alpha\in \Sigma(1)})$, where $V$ is a finite dimension $k$-vector space, and where each $F^\bullet_\alpha$ is a finite decreasing filtration on $V$. Morphisms in $\Filt^{\Sigma(1)}(k)$ are $k$-linear maps that preserve the filtrations.

\begin{thm}
\label{thm:Klyachko}
  Let $(S,s)$ be a toric $\A$-variety.  The functor ``fiber over $s$" induces an equivalence between the category $\Vec^\A(S)$ and the full subcategory of $\Filt^{\Sigma(1)}(k)$ spanned by objects $(V,\{F_\alpha^{\bullet}\}_{\alpha \in \Sigma(1)})$ satisfying the following compatibility condition among the filtrations:
\begin{itemize}
  \item[\bf{(K)}] for each $\sigma \in \Sigma$, there is a $\mathrm{X}^*(\A)$-grading $V = \bigoplus_{\chi \in \mathrm{X}^*(A)} V^{\sigma}_{\chi}$ such that
\[
      F^p_\alpha(V) = \bigoplus_{\{\chi | \langle \chi,n_{\alpha} \rangle \geq p \}} V^{\sigma}_{\chi}
\]
for all $p\in\Z$ and $\alpha\in\Sigma(1)$ such that $\alpha$ is a face of $\sigma$.
\end{itemize}
\end{thm}

Composing the restriction functor along $\A\ltimes S\to \A\ltimes \N$ with Klyachko's fiber-at-$s$ functor of Theorem \ref{thm:Klyachko} defines a functor $\Vec(\A\ltimes\N)\to \Filt^{\Sigma(1)}(k)$. On the other hand, we have $(\A\ltimes \N)|_s = (\G\ltimes X)|_x = \G_x\ltimes \Spec(k)$, and so restriction to the fiber at $s$ defines a functor $\Vec(\A\ltimes\N)\to \Rep(\G_x)$. Combining these two constructions, we obtain a functor
\[
\Phi_{\G,X}: \Vec(\A\ltimes \N) \to \Filt^{\G_x,\Sigma(1)}(k).
\]
Here $\Filt^{\G_x,\Sigma(1)}(k)$ has as objects collections $(V,\rho,\{F_i^{\bullet}\}_{i \in \Sigma(1)})$, where $(V,\rho)$ is an object of $\Rep(\G_x)$, and where $(V,\{F_i^{\bullet}\}_{i \in \Sigma(1)})$ is an object of $\Filt^{\Sigma(1)}(k)$. Morphisms are $\G_x$-equivariant $k$-linear maps that preserve the filtrations.

Since $\N$ is a smooth and $\A$-equivariant group scheme over $S$, its Lie algebra $\Lie(\N)$ comes equipped with the structure of an $\A$-equivariant vector bundle over $S$. Thus Klyachko's functor equips the fiber of $\Lie(\N)$ over $S$ with the structure of a $\Sigma(1)$-filtered vector space. The fiber of $\N$ over $s$ is $\G_x$, and so $\Lie(\G_x)$ becomes a $\Sigma(1)$-filtered vector space (and even a $\Sigma(1)$-filtered Lie algebra). Our generalization of Theorem \ref{thm:example} is as follows.

\begin{thm}
\label{thm:main}
Let $(X,x)$ be a (split) neutralizable toroidal $(\G,\B)$-spherical variety. The functor $\Phi_{\G,X}:\Vec^\G(X)\to \Filt^{\G_x,\Sigma(1)}(k)$ is fully faithful, and its essential image is spanned by objects $(W,\rho,\{F_\alpha^{\bullet}\}_{\alpha \in \Sigma(1)})$ satisfying Klyachko's condition {\bf{(K)}} and the following transversality condition:
\begin{itemize}
\item[{\bf{(C)}}] the action of the $\Sigma(1)$-filtered Lie algebra $\Lie(\G_x)$ on $W$ via $d\rho$ respects the filtrations, i.e.,  for all $i,j\in \Z$ and $\alpha\in \Sigma(1)$, we have
    \[
    F_\alpha^i(\Lie(\G_x))\times F_\alpha^j(W)\to F_\alpha^{i+j}(W).
    \]
\end{itemize}
\end{thm}

\section{Final Remarks}\label{s:remarks}
To close, we connect our results more precisely with those in the literature; at the same time, we give examples to show that the neutralizibility condition is both interesting and non-vacuous.  Given a toroidal spherical $(\G,\B)$-variety $(X,x)$, recall the neutralizability only depends on the homogeneous spherical variety $(X^o,x)$. For this reason, in most of the examples, we describe only a homogeneous spherical variety and its weight lattice. In the neutralizable cases, our results give descriptions of equivariant vector bundles on any toroidal spherical variety containing the given $(X^o,x)$ as its open $\G$-orbit. In the examples below, $k$ is an algebraically closed field of characteristic $0$. (In these specific cases, however, more care in the formulations would allow one to work without assuming $k$ to be algebraically closed.)

\begin{ex}
Let $V$ be an $(n+1)$-dimensional $k$-vector space, and let $V^{\vee}$ be its $k$-vector-space dual.  Consider the diagonal action of $\G = \PGL(V)$ on $X = {\mathbb P}(V) \times {\mathbb P}(V^{\vee})$. We write $(v,\lambda)$ for the bihomogeneous coordinates on $X$.
The closed subscheme of $X$ defined in bihomogeneous coordinates by $\lambda(v) = 0$ is $\G$-stable, and its complement $X^o$ is homogeneous for the action of $\G$.
Fix a complete flag $0 = V_0\subset \dotsb \subset V_{n+1}$ on $V$, and let $B$ be its stabilizer in $\G$. Let $x = (v,\lambda)\in X^o(k)$ with $\lambda(V_1)\neq 0$ and $v \not\in V_n$. Then $(X,x)$ is $(\G,\B)$-spherical.

Let $\B'$ be the stabilizer of the flag in $\GL(V)$. For $0\leq i\leq n$, let $\chi_i:\B'\to \gm$ be the character through which $B'$ acts on $V_{i+1}/V_i$. The characters $\chi_i$ form a basis for $\mathrm{X}^*(\B')$. We identify $\mathrm{X}^*(\B)$ with the subgroup of $\mathrm{X}^*(\B')$ composed of elements whose components with respect to the $\chi_i$ sum to $0$. The weight lattice $\Lambda$ of $(X,x)$ is then generated by $\chi_0 \chi_n^{-1}$. From this description, it is clear that $\Lambda$ is a direct summand of $\mathrm{X}^*(\B)$, and so $(X,x)$ is neutralizable. This class of neutralizable examples generalizes our motivating example, and our methods describe the category of $\mathrm{PGL}(V)$-equivariant vector bundles on $X$.
%
\end{ex}

\begin{ex}
Let $\H$ be a connected reductive group over $k$.  Let $\G = \H\times \H$, which acts on $\H$ by $(g,g').h\mapsto ghg'^{-1}$. Let $\T$ be a maximal torus of $\H$, and let $\B$ and $\B'$ be opposed Borel subgroups of $\H$ containing $\T$. Let $e\in \H(k)$ be the identity section. Then $(\H,e)$ is a $(\G, \B\times \B')$-spherical variety, as $\B.\B'$ is open in $\H$ because $\B$ and $\B'$ are opposed. The stabilizer of $e$ in $\G$ is $\H$ under its diagonal embedding. The intersection $\H\cap (\B\times \B')$ is the diagonal copy of $\T$, since $\B$ and $\B'$ are opposed. Since $\T\times \T$ splits as the product of the diagonal copy of $\T$ and $\T\times 1$, the spherical variety $(\H,e)$ is neutralizable, and our methods describe the category of equivariant vector bundles on any toroidal partial compactification of $(\H,e)$. In particular, we recover the results of Kato ({\em cf.} \cite{Kato}).
\end{ex}

\begin{ex}
Again, let $\G$ be a connected reductive group over $k$.  Let $\B$ be a Borel subgroup of $\G$, and let $\U$ be the unipotent radical of a Borel subgroup opposed to $\B$. Let $X = \G/\U$, and let $x\in X(k)$ correspond to the identity coset. Then $(X,x)$ is a $(\G,\B)$-spherical homogeneous space. The weight lattice of $(X,x)$ is the full character lattice of $\B$, since $\B\cap \U = \langle e \rangle$. Thus $(X,x)$ is neutralizable, and our methods describe the category of equivariant vector bundles on any toroidal partial compactification of $(X,x)$.
\end{ex}

\begin{nex}
It is easy to produce non-neutralizable toroidal spherical varieties. A simple case is $\gm$ acting on itself by the character $t\mapsto t^2$, for which the weight lattice of the spherical variety has index $2$ in the character lattice of the group. Although this spherical variety is not neutralizable, a simple variant on the Rees construction (Theorem \ref{thm:rees}) provides a description of the equivariant vector bundles on partial compactifications. Another similar example is a variant on the example of \S\ref{s:example}. Let $\mathrm{N}$ be the normalizer of the $\T$ from \S\ref{s:example} in $\G = \PGL_2$. Let $X = \G/\mathrm{N}$, and let $x\in X(k)$ correspond to the identity coset. Then for $\B$ as in \S\ref{s:example}, one finds that $(X,x)$ is $(\G,\B)$-spherical. The weight lattice of $(X,x)$ has index $2$ in the character lattice of $\B$, and so $(X,x)$ is not neutralizable. Nonetheless, one can still describe the equivariant vector bundles on partial compactifications using our techniques using the modified Rees construction.

More generally, one can analyze many non-neutralizable examples using a modified Klyachko construction and our semi-direct product method. We do not have a systematic description of the scope of this more-general approach.
\end{nex}

\begin{footnotesize}
\bibliographystyle{alpha}
\bibliography{announcement}
\end{footnotesize}
\end{document}